\documentclass[12pt]{article}
\usepackage{amsmath,amssymb,amsthm, mathrsfs}
\usepackage{latexsym}
\usepackage{hyperref}
\usepackage[bf, small]{titlesec}
\usepackage{authblk} 
\usepackage{booktabs}
\usepackage{array}

\theoremstyle{plain}
\newtheorem{Thm}{Theorem}[section]
\newtheorem{Lem}[Thm]{Lemma}
\newtheorem{Prop}[Thm]{Proposition}
\newtheorem{Cor}[Thm]{Corollary}

\theoremstyle{definition}

\usepackage{standalone}
\usepackage{tikz}
\usetikzlibrary{arrows,positioning,matrix,fit,backgrounds,shapes,shapes.geometric, intersections,decorations.pathreplacing,calc}

\usetikzlibrary{shapes.callouts,decorations.pathmorphing} 
\usetikzlibrary{shadows} 
\usepackage{calc} 
\usetikzlibrary{decorations.markings} 
\tikzstyle{vertex}=[circle, draw, inner sep=0pt, minimum size=6pt] 

\usetikzlibrary{arrows,matrix} 

\title{On the limit of the sequence $\left\{ C^m(D) \right\}_{m=1}^{\infty}$ for a multipartite tournament $D$}

\author[1]{\small Ji-Hwan Jung}
\author[2]{\small Suh-Ryung Kim}
\author[2]{\small Hyesun Yoon}
\affil[1]{\footnotesize Department of Mathematics Education, Chinju National University of Education, Jinju 52673}
\affil[2]{\footnotesize Department of Mathematics Education, Seoul National University, Seoul 08826}
\affil[ ]{\footnotesize\textit{jhjung@cue.ac.kr, srkim@snu.ac.kr, magisakura@snu.ac.kr}}
\date{}
\begin{document}
\maketitle
\begin{abstract}
For an integer $k \ge 2$, let $A$ be a Boolean block matrix with blocks $A_{ij}$ for $1 \le i,j \le k$ such that $A_{ii}$ is a zero matrix and $A_{ij}+A_{ji}^T$ is a matrix with all elements $1$ but not both corresponding elements of $A_{ij}$ and $A_{ji}^T$ equal to $1$ for $i \neq j$.

Jung~{\em et al.} [Competition periods of multipartite tournaments. {\it Linear and Multilinear Algebra}, https://doi.org/10.1080/03081087.2022.2038057] studied the matrix sequence $\{A^m(A^T)^m\}_{m=1}^{\infty}$.
This paper, which is a natural extension of the above paper and was initiated by the observation that $\{A^m(A^T)^m\}_{m=1}^{\infty}$ converges if $A$ has no zero rows, computes the limit of the matrix sequence $\{A^m(A^T)^m\}_{m=1}^{\infty}$ if $A$ has no zero rows.
To this end, we take a graph theoretical approach:
noting that $A$ is the adjacency matrix of a multipartite tournament $D$, we compute the limit of the graph sequence $\left\{ C^m(D) \right\}_{m=1}^{\infty}$ when $D$ has no sinks.
\end{abstract}


\section{Introduction}
The underlying graph of each digraph in this paper is assumed to be simple unless otherwise mentioned.

We use the notation $u \to v$ for ``there is an arc $(u,v)$'' and the notation $X \to Y$ for ``there is an arc $(x,y)$ for each $x \in X$ and for each $y \in Y$'' in the given digraph.
For simplicity, we omit the braces if $X$ or $Y$ is a singleton set.

Given a digraph $D$ and a positive integer $m$, a vertex $y$ is an {\em $m$-step prey} of a vertex $x$ if and only if there exists a directed walk from $x$ to $y$ of length $m$.
Given a positive integer $m$, the {\em $m$-step competition graph} of a digraph $D$, denoted by $C^m(D)$, has the same vertex set as $D$ and has an edge between vertices $u$ and $v$ if and only if there exists an $m$-step common prey of $u$ and $v$ in $D$.
The notion of an $m$-step competition graph introduced by Cho~{\em et al.}~\cite{cho2000m} is a generalization of the competition graph introduced by Cohen~\cite{cohen1968interval} (the {\em competition graph} of a digraph $D$ is $C^1(D)$).
Since its introduction, an $m$-step competition graph  has been extensively studied (see, for example, \cite{belmont2011complete,helleloid2005connected,ho2005m,park2011m,zhao2009note}).

For the two-element Boolean algebra $\mathcal{B}=\{0,1\}$, $\mathcal{B}_n$ denotes the set of all $n \times n$ matrices over $\mathcal{B}$.
Under the Boolean operations ($1 + 1 = 1$, $0 + 0 = 0$, $1 + 0 = 1$, $1 \times 1 = 1$, $0 \times 0 = 0$, $1 \times 0 = 0$), matrix addition and multiplication are still well-defined in $\mathcal{B}_n$.
Throughout this paper, a matrix is Boolean unless otherwise mentioned.

We note that the adjacency matrix of $C^m(D)$ for a digraph $D$ of order $n$ is the matrix $A^*_{m}$ obtained from $A^m(A^T)^m$ by replacing each of diagonal elements with $0$ where $A$ is the adjacency matrix of $D$.
To see why, we take two distinct vertices $u$ and $v$ of $D$ and suppose that the $i$th row and the $j$th row of $A$ are the rows corresponding to $u$ and $v$, respectively.
Then
\begin{tabbing}
\ \ \ \ \ \ \= $u$ and $v$ are adjacent in $C^m(D)$ \\
$\Leftrightarrow$ \>  $u$ and $v$ have an $m$-step common prey in $D$ \\
$\Leftrightarrow$ \> the inner product of the $i$th row and the $j$th row of $A^m$ is $1$\\
$\Leftrightarrow$ \>  the $(i,j)$-entry of $A^*_m$ is $1$.
\end{tabbing}
It is easy to check that $A^*_i=A^*_j$ if and only if $A^i(A^T)^i=A^j(A^T)^j$ for a $(0,1)$ Boolean matrix $A$ of order $n$ and any positive integers $i$ and $j$.
Therefore, for a digraph $D$ and its adjacency matrix $A$,
\begin{equation}\label{eq:iff}
C^i(D)=C^j(D) \Leftrightarrow A^i(A^T)^i=A^j(A^T)^j
\end{equation}
for any positive integers $i$ and $j$.

A \emph{$k$-partite tournament} is an orientation of a complete $k$-partite graph for a positive integer $k$.
	The adjacency matrix of a $k$-partite tournament for $k \ge 2$ is represented as a block matrix $A$ with blocks $A_{ij}$ for $1 \le i,j \le k$ such that $A_{ii}$ is a zero matrix and $A_{ij}+A_{ji}^T$ is a matrix with all elements $1$ but not both corresponding elements of $A_{ij}$ and $A_{ji}^T$ equal to $1$ for $i \neq j$ (see $D$ and $A$ in Figure~\ref{fig:multipartitematrix} for an example).

We call a $k$-partite tournament a {\em multipartite tournament} if $k \ge 2$.

Given a Boolean matrix $A$ of order $n$, consider the matrix sequence $\{A^m(A^T)^m\}_{m=1}^{\infty}$.
Since the cardinality of the Boolean matrix set $\mathcal{B}_n$ is equal to $2^{n^2}$ which is finite, there is the smallest positive integer $q$ such that, for some positive integer $r$, $A^{q+i}(A^T)^{q+i}=A^{q+r+i}(A^T)^{q+r+i}$ for every nonnegative integer $i$.
Then there is also the smallest positive integer $p$ such that $A^{q}(A^T)^q=A^{q+p}(A^T)^{q+p}$.
Those integers $q$ and $p$ are called the {\it competition index} and the {\it competition period} of $A$, respectively, which  Cho and Kim~\cite{cho2013competition} introduced.
If $D$ is the digraph whose adjacency matrix is $A$, then, by \eqref{eq:iff}, $q$ and $p$ are the smallest integers such that, for some positive integer $r$, $C^{q+i}(D)=C^{q+r+i}(D)$ for all nonnegative integers $i$ and $C^{q}(D)=C^{q+p}(D)$, respectively.
In this respect, $q$ and $p$ are called the \emph{competition index} and the \emph{competition period} of $D$, respectively, and denoted by ${\rm cindex}(D)$ and ${\rm cperiod}(D)$, respectively, which Cho and Kim~\cite{cho2004competition} introduced.
Refer to \cite{kim2010generalized,kim2015characterization} for some results of competition indices and competition periods of digraphs.

Jung~{\em et al.}~\cite{jung2022competition} computed the competition period of a
multipartite tournament.

\begin{Thm}[\cite{jung2022competition}]\label{thm:main}
The competition period of a $k$-partite tournament $D$ is at most three for any integer $k\ge 2$.
Further, if $k=2$, then the competition period of $D$ is at most two.
\end{Thm}

\begin{figure}
\begin{center}
\begin{minipage}[c]{.45\textwidth}
\begin{center}
\begin{tikzpicture}[auto,thick,scale=1]
    \tikzstyle{player}=[minimum size=5pt,inner sep=0pt,outer sep=0pt,ball color=black,circle]
    \tikzstyle{source}=[minimum size=5pt,inner sep=0pt,outer sep=0pt,ball color=black, circle]
    \tikzstyle{arc}=[minimum size=5pt,inner sep=1pt,outer sep=1pt, font=\footnotesize]
    \path (150:2cm)  node [player, label=left:$v_2$]  (1) {};
    \path (180:2cm)  node [player, label=left:$v_3$]  (2) {};
    \path (210:2cm)  node [player, label=left:$v_4$]  (3) {};
    \path (310:2cm)  node [player, label=right:$v_5$]  (4) {};
    \path (355:1.5cm)  node [player, label=right:$v_6$]  (5) {};
    \path (80:2cm)  node [player, label=above:$v_1$]  (6) {};
\draw[black,thick,-stealth] (1) - +(5);
\draw[black,thick,-stealth] (3) - +(4);
\draw[black,thick,-stealth] (4) - +(1);
\draw[black,thick,-stealth] (4) - +(2);
\draw[black,thick,-stealth] (2) - +(5);
\draw[black,thick,-stealth] (5) - +(3);
\draw[black,thick,-stealth] (6) - +(1);
\draw[black,thick,-stealth] (6) - +(2);
\draw[black,thick,-stealth] (6) - +(3);
\draw[black,thick,-stealth] (6) - +(4);
\draw[black,thick,-stealth] (6) - +(5);
\draw (0,-3) node{$D$};
\draw[dotted, thick] (-1.8, 0) ellipse (1 and 1.8);
\draw[dotted, thick] (1.3, -0.7) ellipse (1 and 1.3);
\draw[dotted, thick] (0.3, 2) ellipse (0.7 and 0.7);

    \end{tikzpicture}
\end{center}
\end{minipage}
\begin{minipage}[c]{.45\textwidth}
\begin{tikzpicture}
		[every matrix/.style={column sep = {1.5em,between origins},row sep={1.5em,between origins}}]
		\matrix (m) [matrix of math nodes,ampersand replacement=\&, left delimiter=(,right delimiter=)]
		{
			{0} \& {1} \& {1} \& {1} \& {1} \& {1} \\
			{0} \& {0} \& {0} \& {0} \& {0} \& {1} \\
			{0} \& {0} \& {0} \& {0} \& {0} \& {1} \\
			{0} \& {0} \& {0} \& {0} \& {1} \& {0} \\
			{0} \& {1} \& {1} \& {0} \& {0} \& {0} \\
			{0} \& {0} \& {0} \& {1} \& {0} \& {0} \\
		};
		
		\matrix (m2) [matrix of nodes,ampersand replacement=\&, left=0.4em of m]
		{
			{$v_1$} \\
			{$v_2$} \\
			{$v_3$} \\
			{$v_4$} \\
			{$v_5$} \\
			{$v_6$} \\
		};
		
		\matrix (m3) [matrix of nodes,ampersand replacement=\&, above=-0.4em of m]
		{
			{$v_1$} \& {$v_2$} \& {$v_3$} \& {$v_4$} \& {$v_5$} \& {$v_6$} \\
		};
\draw (0,-2.8) node{$A$};
\draw[-] ($0.5*(m-1-1.north east)+0.5*(m-1-2.north west)$) -- ($0.5*(m-6-1.south east)+0.5*(m-6-2.south west)$);
\draw[-] ($0.5*(m-1-1.south west)+0.5*(m-2-1.north west)$) -- ($0.5*(m-1-6.south east)+0.5*(m-2-6.north east)$);
\draw[-] ($0.5*(m-1-4.north east)+0.5*(m-1-5.north west)$) -- ($0.5*(m-6-4.south east)+0.5*(m-6-5.south west)$);
\draw[-] ($0.5*(m-4-1.south west)+0.5*(m-5-1.north west)$) -- ($0.5*(m-4-6.south east)+0.5*(m-5-6.north east)$);
\end{tikzpicture}
\end{minipage}
\end{center}
\caption{A tripartite tournament $D$ and its adjacency matrix $A$ where $V_1 =\{v_1\}$, $V_2 =\{v_2,v_3,v_4\}$, and $V_3=\{v_5, v_6\}$ are the partite sets of $D$.}
\label{fig:multipartitematrix}
\end{figure}
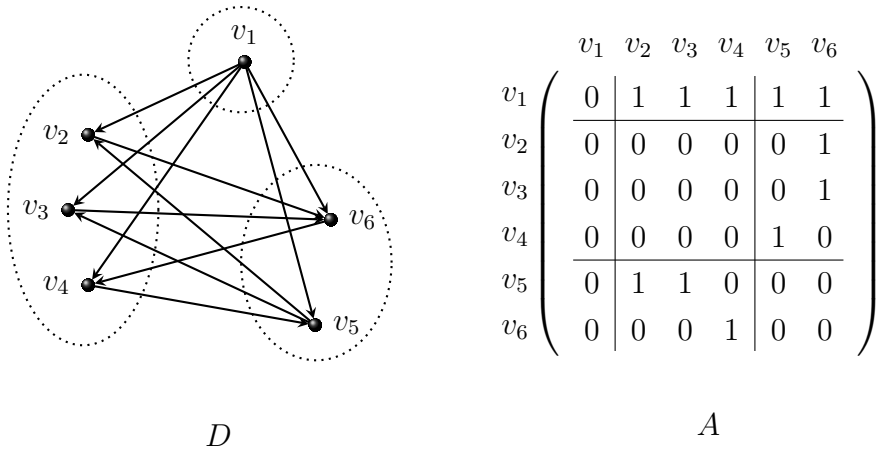

As we observed above, they actually computed the competition period of a Boolean block matrix $A$ with blocks $A_{ij}$ for $1 \le i,j \le k$ such that $A_{ii}$ is a zero matrix and $A_{ij}+A_{ji}^T$ is a matrix with all elements $1$ but not both corresponding elements of $A_{ij}$ and $A_{ji}^T$ equal to $1$ for $i \neq j$.
In this paper, we compute the limit of the matrix sequence $\{A^m(A^T)^m\}_{m=1}^{\infty}$, if it exists, as a natural extension of the result obtained by Jung~{\em et al.}\cite{jung2022competition}.
To this end, we investigate the limit of $\{C^m(D)\}_{m=1}^{\infty}$ for the digraph $D$  of $A$, which is a multipartite tournament.

The greatest common divisor of the lengths of the closed directed walks of a strongly connected nontrivial digraph $D$ is called the {\it index of imprimitivity} of $D$ and denoted by $\kappa(D)$.
If $\kappa(D)=1$, then $D$ is said to be {\em primitive}.
If $A$ is the adjacency matrix of a primitive digraph $D$, then there exists a positive integer $M$ such that $A^m$ becomes a matrix of all $1$s for any integer $m \ge M$.
We call the smallest $M$ the \textit{exponent} of $D$ and denote it by $\exp(D)$.
The vertex set of $D$ can be partitioned into $\kappa(D)$ nonempty subsets $U_1, U_2, \ldots, U_{\kappa(D)}$ where each arc of $D$ goes out from $U_i$ and enters $U_{i+1}$ (identify $U_{\kappa(D)+1}$ with $U_1$) for some $i \in \{1 , \ldots, \kappa(D)\}$ (see \cite{brualdi1991combinatorial}).
We call the sets $U_1, U_2, \ldots, U_{\kappa(D)}$ the {\it sets of imprimitivity} of $D$.

Suppose that $D$ is a weakly connected digraph.
Then the strong components of $D$ can be arranged as $Q_1,\ldots, Q_s$ so that there is no arc going from $Q_i$ to $Q_j$ whenever $i > j$.
We call the strong components arranged in such a way
\textit{ordered strong components} of $D$.
For convenience, we call $Q_s$ the {\it last strong component} of $D$.
We denote the sets of imprimitivity of $Q_i$ by
\[
U_1^{(i)}, U_2^{(i)}, \ldots, U_{\kappa(Q_i)}^{(i)}
\]
for each integer $i=1,\ldots,s$.

For the digraph $D$ given in Figure~\ref{fig:multipartitematrix},
the subdigraphs $Q_1$ and $Q_2$ induced by $V_1$ and $V_2 \cup V_3$, respectively, form the ordered strong components of $D$.
It is easy to check that
$\kappa(Q_2)=4$, $U_{1}^{(2)}=\{v_2,v_3\}$,
$U_{2}^{(2)}=\{v_6\}$,
$U_{3}^{(2)}=\{v_4\}$,
and $U_{4}^{(2)}=\{v_5\}$.
Further,
\[
\begin{tikzpicture}
\matrix [matrix of math nodes,left delimiter=(,right delimiter=),row sep=0.1cm,column sep=0.1cm] (m) {
       1 & 1 & 1 & 1 & 1 & 1 \\
  1 & 1 & 1 & 0 & 0 & 0 \\
  1 & 1 & 1 & 0 & 0 & 0 \\
  1 & 0 & 0 & 1 & 0 & 0 \\
  1 & 0 & 0 & 0 & 1 & 0 \\
  1 & 0 & 0 & 0 & 0 & 1\\ };

\draw[-] ($0.5*(m-1-1.north east)+0.5*(m-1-2.north west)$) -- ($0.5*(m-6-1.south east)+0.5*(m-6-2.south west)$);
\draw[-] ($0.5*(m-2-3.north east)+0.5*(m-2-4.north west)$) -- ($0.5*(m-4-3.south east)+0.5*(m-4-4.south west)$);
\draw[-] ($0.5*(m-4-4.north east)+0.5*(m-4-5.north west)$) -- ($0.5*(m-5-4.south east)+0.5*(m-5-5.south west)$);
\draw[-] ($0.5*(m-5-5.north east)+0.5*(m-5-6.north west)$) -- ($0.5*(m-6-5.south east)+0.5*(m-6-6.south west)$);

\draw[-] ($0.5*(m-1-1.south west)+0.5*(m-2-1.north west)$) -- ($0.5*(m-1-6.south east)+0.5*(m-2-6.north east)$);
\draw[-] ($0.5*(m-3-2.south west)+0.5*(m-4-2.north west)$) -- ($0.5*(m-3-4.south east)+0.5*(m-4-4.north east)$);
\draw[-] ($0.5*(m-4-4.south west)+0.5*(m-5-4.north west)$) -- ($0.5*(m-4-5.south east)+0.5*(m-5-5.north east)$);

\draw[-] ($0.5*(m-5-5.south west)+0.5*(m-6-5.north west)$) -- ($0.5*(m-5-6.south east)+0.5*(m-6-6.north east)$);
\draw (-3.5,0) node{$A^m (A^T)^m =$};
\end{tikzpicture}
\]
for any positive integer $m$ where $A$ is the adjacency matrix of $D$.
Therefore the above matrix is the  limit of the matrix sequence $\{A^m(A^T)^m\}_{m=1}^{\infty}$.
We note that it corresponds to $M_2$ in Figure~\ref{fig:graphs} where $J^{(2)}$ and $J^{(3)}$ do not exist.
This is not a coincidence as shown by Corollary~\ref{cor:last}.

Throughout this paper, we take a graph theoretical approach to derive the following result whose matrix version is Corollary~\ref{cor:last}.

\begin{Thm}\label{thm:last}
Suppose that a $k$-partite tournament $D$ with $k$-partition $(V_1,V_2,\ldots,V_k)$ has ordered strong components $Q_1, Q_2, \ldots, Q_s$ for some integers $k \ge 2$ and $s \ge 1$.
Then, if $Q_s$ is nontrivial, the graph sequence $\{C^m(D)\}_{m=1}^{\infty}$ converges to a graph
\[
G \cong \begin{cases}
          K_{|V(D)|}, & \mbox{if } \kappa(Q_s)=1; \\
          G_1, & \mbox{if } \kappa(Q_s)=2; \\
          G_3, & \mbox{if } \kappa(Q_s)=3; \\
          G_2, & \mbox{if } \kappa(Q_s)=4,
        \end{cases}
\]
 where $G_1$, $G_2$, and $G_3$ are the graphs given in Figure~\ref{fig:graphs}.
\end{Thm}

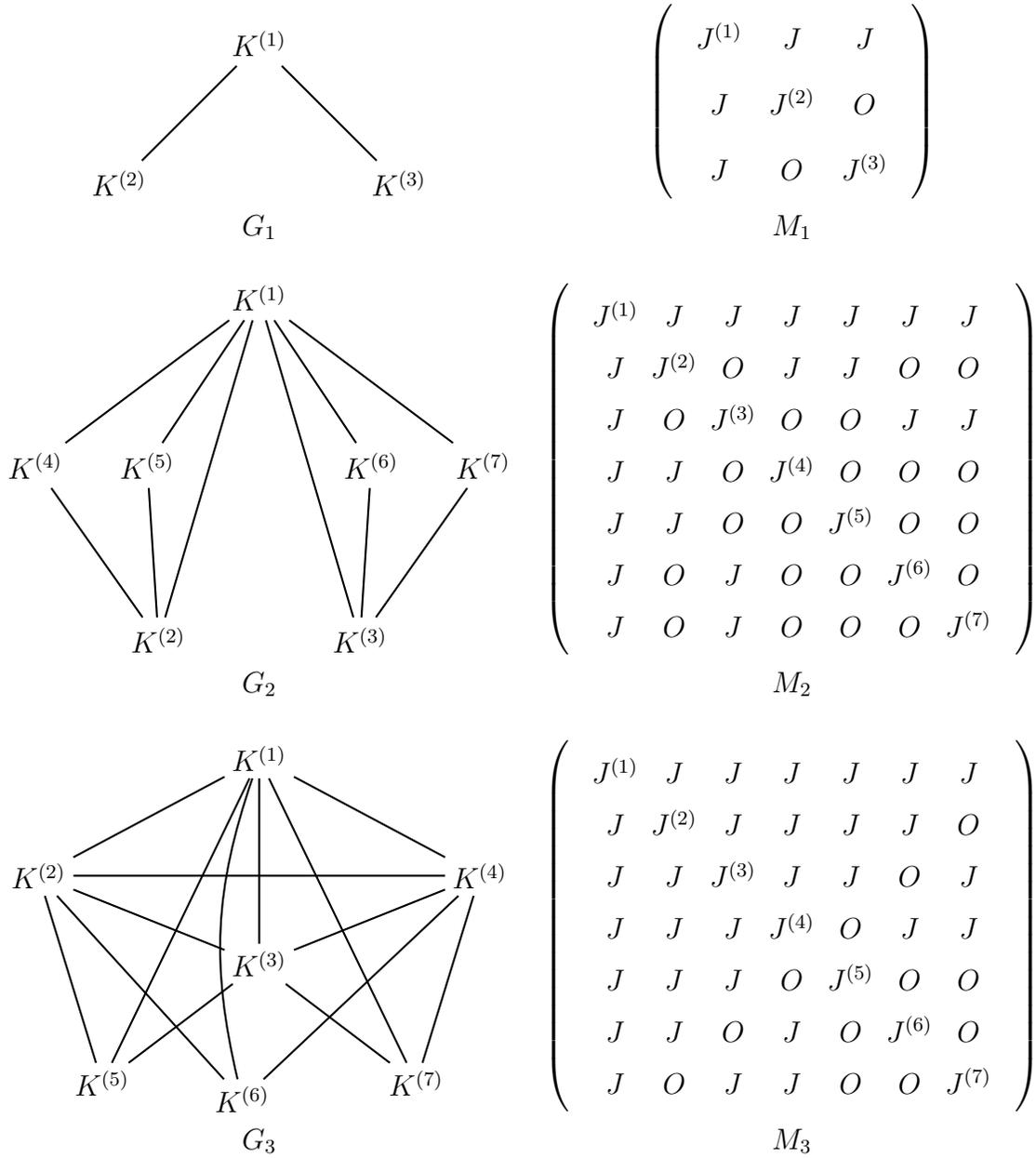
\begin{figure}
  \begin{center}
  \begin{tabular}{cc}
\begin{tikzpicture}[auto,thick,scale=1]
    \tikzstyle{player}=[minimum size=5pt,inner sep=1pt,outer sep=0pt,circle,draw=black]
    \tikzstyle{player1}=[minimum size=5pt,inner sep=1pt,outer sep=3pt,rectangle,draw=white]
    \tikzstyle{source}=[minimum size=5pt,inner sep=0pt,outer sep=0pt,ball color=black, circle]
    \tikzstyle{arc}=[minimum size=5pt,inner sep=1pt,outer sep=1pt, font=\footnotesize]
    \path (90:2cm)     node [player1]  (a) {\phantom{d}$K^{(1)}$\phantom{d}};
    \path (0:2cm)       node [player1] (b){\phantom{dd}$K^{(3)}$\phantom{dd}};
    \path (180:2cm)       node [player1] (c){\phantom{dd}$K^{(2)}$\phantom{dd}};
    \draw[black,thick,-] (a) -- ++(b);
    \draw[black,thick,-] (a) -- ++(c);

    \end{tikzpicture}
    &
\begin{tikzpicture}
\matrix [matrix of math nodes,left delimiter=(,right delimiter=),row sep=0.3cm,column sep=0.1cm] (m) {
      J^{(1)} & J &  J \\
      J & J^{(2)} &  O  \\
      J & O   &J^{(3)}\\ };
\end{tikzpicture}
\\
    $G_1$ & $M_1$\\
    &\\

    \begin{tikzpicture}[auto,thick,scale=0.8]
    \tikzstyle{player}=[minimum size=5pt,inner sep=1pt,outer sep=0pt,circle,draw=black]
    \tikzstyle{player1}=[minimum size=5pt,inner sep=1pt,outer sep=3pt,rectangle,draw=white]
    \tikzstyle{source}=[minimum size=5pt,inner sep=0pt,outer sep=0pt,ball color=black, circle]
    \tikzstyle{arc}=[minimum size=5pt,inner sep=1pt,outer sep=1pt, font=\footnotesize]
    \path (90:3cm)     node [player1]  (a) {$K^{(1)}$};
    \path (300:3.6cm)       node [player1] (b){$K^{(3)}$};
    \path (240:3.6cm)       node [player1] (c){$K^{(2)}$};
    \path (0:2cm)       node [player1] (d){$K^{(6)}$};
    \path (0:4cm)       node [player1] (e){$K^{(7)}$};
    \path (180:2cm)       node [player1] (f){$K^{(5)}$};
    \path (180:4cm)       node [player1] (g){$K^{(4)}$};

    \draw[black,thick,-] (a) -- ++(b);
    \draw[black,thick,-] (a) -- ++(c);
    \draw[black,thick,-] (b) -- ++(d);
    \draw[black,thick,-] (b) -- ++(e);
    \draw[black,thick,-] (c) -- ++(f);
    \draw[black,thick,-] (c) -- ++(g);
    \draw[black,thick,-] (a) -- ++(f);
    \draw[black,thick,-] (a) -- ++(g);
    \draw[black,thick,-] (a) -- ++(d);
    \draw[black,thick,-] (a) -- ++(e);

    \end{tikzpicture}
    & \begin{tikzpicture}
\matrix [matrix of math nodes,left delimiter=(,right delimiter=),row sep=0.1cm,column sep=-0.1cm] (m) {
      J^{(1)} & J & J  & J & J  & J & J  \\
      J & J^{(2)} & O  & J & J  & O & O  \\
      J & O   &J^{(3)} & O & O  & J & J  \\
      J & J & O & J^{(4)} & O & O & O \\
      J & J & O &O   &J^{(5)}& O & O  \\
      J &  O & J  &  O   & O  & J^{(6)} & O  \\
      J & O & J & O & O & O & J^{(7)}\\ };
\end{tikzpicture}\\
$G_2$ & $M_2$\\
    &\\

    \begin{tikzpicture}[auto,thick,scale=0.7]
    \tikzstyle{player}=[minimum size=5pt,inner sep=1pt,outer sep=0pt,circle,draw=black]
    \tikzstyle{player1}=[minimum size=5pt,inner sep=0pt,outer sep=3pt,rectangle,draw=white]
    \tikzstyle{source}=[minimum size=5pt,inner sep=0pt,outer sep=0pt,ball color=black, circle]
    \tikzstyle{arc}=[minimum size=5pt,inner sep=1pt,outer sep=1pt, font=\footnotesize]
    \path (30:5.2cm)  node [player1]  (a) {$K^{(4)}$};
    \path (90:0.8cm)     node [player1]  (b) {$K^{(3)}$};
    \path (150:5.2cm)  node [player1]  (c) {$K^{(2)}$};
    \path (333:3.6cm)   node [player1]  (d) {$K^{(7)}$};
    \path (260:2cm)       node [player1]  (e){$K^{(6)}$};
    \path (207:3.6cm)  node [player1]  (f){$K^{(5)}$};
    \path (90:5cm)  node [player1]  (g){$K^{(1)}$};

    \draw[black,thick,-] (a) -- ++(b);
    \draw[black,thick,-] (b) -- ++(c);
   \draw[black,thick,-] (a) -- ++(c);
    \draw[black,thick,-] (a) -- ++(e);
    \draw[black,thick,-] (c) -- ++(e);
    \draw[black,thick,-] (b) -- ++(f);
    \draw[black,thick,-] (b) -- ++(d);
    \draw[black,thick,-] (c) -- ++(f);
    \draw[black,thick,-] (a) -- ++(d);
    \draw[black,thick,-] (g) -- ++(a);
    \draw[black,thick,-] (g) -- ++(b);
    \draw[black,thick,-] (g) -- ++(c);
    \draw[black,thick,-] (g) -- ++(d);
   \path (g) edge [black, bend right=16,thick,-] (e);
    \draw[black,thick,-] (g) -- ++(f);
    \end{tikzpicture}
         &
        \begin{tikzpicture}
\matrix [matrix of math nodes,left delimiter=(,right delimiter=),row sep=0.1cm,column sep=-0.1cm] (m) {
      J^{(1)} & J &   J  & J & J   & J &J \\
      J & J^{(2)} &   J  & J & J   & J & O  \\
      J & J   &J^{(3)} & J & J   & O & J  \\
      J & J & J & J^{(4)} & O & J & J \\
      J & J    & J    &   O   &J^{(5)}& O & O  \\
      J &  J    & O    &  J   & O  & J^{(6)} &O  \\
      J & O & J & J & O & O & J^{(7)} \\ };
\end{tikzpicture}\\
        $G_3$ & $M_3$

%
\end{tabular}
\end{center}
\caption{The graphs $G_1$, $G_2$, and $G_3$ in Theorem~\ref{thm:last}.
        In each graph, $K^{(i)}$ either does not exist or stands for a clique.
       The line between two cliques $K^{(i)}$ and $K^{(j)}$ indicates that $K^{(i)} \cup K^{(j)}$ forms a clique while the absence of a line between $K^{(i)}$ and $K^{(j)}$ means that there is no edge joining a vertex in $K^{(i)}$ and a vertex in $K^{(j)}$.
       The matrix $M_i= A(G_i)+I$ where $A(G_i)$ is the adjacency matrix of $G_i$ for each $i=1,2,3$; $I$, $J$, and $O$ are an identity matrix, a matrix of all $1$s, and a zero matrix, respectively, of an appropriate order;
       $J^{(s)}$ either does not exist or is a square matrix of all $1$s.
}\label{fig:graphs}
\end{figure}

By Proposition~\ref{Prop:nosinks}, $Q_s$ being nontrivial is equivalent to $D$ having no sinks.
Therefore $Q_s$ being nontrivial is equivalent to the adjacency matrix of $D$ having each row sum nonzero.
By \eqref{eq:iff}, Theorem~\ref{thm:last} may be restated as follows.
\begin{Cor}[Matrix version]\label{cor:last}
For an integer $k \ge 2$, let $A$ be a  block matrix with blocks $A_{ij}$ for $1 \le i,j \le k$ such that $A_{ii}$ is a zero matrix and $A_{ij}+A_{ji}^T$ is a matrix of all $1$s but not both corresponding elements of $A_{ij}$ and $A_{ji}^T$ equal to $1$ for $i \neq j$.
If each row sum of $A$ is nonzero, then the limit of the matrix sequence $\{A^m(A^T)^m\}_{m=1}^{\infty}$ exists and equals one of the matrices given in Figure~\ref{fig:graphs} after a simultaneous permutation of rows and columns.
\end{Cor}

The rest of this paper is devoted to proving Theorem~\ref{thm:last}.
In Section~\ref{sec:2}, we justify that it is sufficient to consider the cases
$\kappa(Q_s)=1,2,3,4$ where $Q_s$ is the last strong component of the given multipartite tournament.
Section~\ref{sec:3} takes care of the cases $\kappa(Q_s)=1,2,4$ while Section~\ref{sec:4} takes care of the case $\kappa(Q_s)=3$.
Finally, Section~\ref{sec:5} winds up the proof of Theorem~\ref{thm:last}.

\section{Preliminaries}\label{sec:2}
\begin{Prop}\label{Prop:nosinks}
  Let $D$ be a multipartite tournament with the last strong component $Q_s$ for some positive integer $s$.
  Then $D$ has no sinks if and only if $Q_s$ is nontrivial.
\end{Prop}

\begin{proof}
  If $Q_s$ is trivial, then the vertex in $Q_s$ is a sink in $D$.
  To show the converse, suppose that $Q_s$ is a nontrivial strong component of $D$ and take a vertex $v \in V(D)$.
  Then there are at least two partite sets of $D$ that intersect with $V(Q_s)$.
   Therefore there exists a partite set $X$ such that $X \cap V(Q_s) \neq \emptyset$ and $v$ does not belong to $X$.
  Then $v$ has an out-neighbor in $X \cap V(Q_s)$ since $D$ is a multipartite tournament.
  Hence $D$ has no sinks.
\end{proof}

The following result given by Eoh~{\em et al.} \cite{eoh2020m} was originally for a bipartite tournament.
Yet, their proof is still valid for a general digraph and so the statement may be restated as follows.

\begin{Prop}[\cite{eoh2020m}]\label{prop:adjacent}
 Let $D$ be a digraph without sinks.
 If two vertices are adjacent in $C^M(D)$ for a positive integer $M$, then they are also adjacent in $C^m(D)$ for any positive integer $m \ge M$.
\end{Prop}

The following is an immediate consequence of the above proposition.

\begin{Cor}\label{cor:nosinks}
  Let $D$ be a digraph without sinks.
  Then the graph sequence $\{C^m(D)\}_{m=1}^{\infty}$ converges.
\end{Cor}

Proposition~\ref{Prop:nosinks} and Corollary~\ref{cor:nosinks} tell us that,
for a multipartite tournament $D$ with the nontrivial last strong component, the graph sequence $\{C^m(D)\}_{m=1}^{\infty}$ converges.
In this vein, throughout this paper, we seek for the limit of $\{C^m(D)\}_{m=1}^{\infty}$
for a multipartite tournament $D$ with the nontrivial last strong component.
We note that $D$ has no sinks by Proposition~\ref{Prop:nosinks}.
Thus,   by Proposition~\ref{prop:adjacent},
\begin{itemize}
  \item[($\ast$)] in order to prove that two vertices $u$ and $v$ are adjacent in the limit of $\{C^m(D)\}_{m=1}^{\infty}$, it is sufficient to show that $u$ and $v$ have an $m$-step common prey for some positive integer $m$.
\end{itemize}

Next, we observe that the limit of $\{C^m(D)\}_{m=1}^{\infty}$ depends on the index of imprimitivity of the last strong component of a multipartite tournament $D$.

\begin{Prop}\label{prop:index}
  Let $D$ be a strongly connected $k$-partite tournament for an integer $k \ge 2$.
  Then the index $\kappa(D)$ of imprimitivity of $D$ is as follows:
  \[
  \kappa(D) = \begin{cases}
                2 \mbox{ or } 4, & \mbox{if } k=2; \\
                1 \mbox{ or } 3, & \mbox{if } k=3; \\
                1, & \mbox{if } k \ge 4.
              \end{cases}
  \]
\end{Prop}

\begin{proof}
It is shown in \cite{bondy1976diconnected} that, for an integer $k \ge 3$, each strongly connected $k$-partite tournament contains a directed cycle of length $\ell$ for each $\ell \in  \{3, 4, \ldots, k\}$.
  Therefore, if $k \ge 4$, then $D$ is primitive.
  Further, if $k=3$, then $D$ contains a directed cycle of length $3$ and so $\kappa(D)$ is $1$ or $3$.

  Suppose $k=2$ and let $(V_1,V_2)$ be a bipartition of $D$.
  Since $D$ is strongly connected, there exists a closed directed walk $W:=u_0 \to u_1 \to \cdots \to u_{l-1} \to u_0$ containing all the vertices of $D$ for some positive integer $l$.
  Since $D$ is a bipartite tournament, the length $l$ of $W$ is even and so $l \ge 4$.
  Without loss of generality, we may assume that a vertex on $W$ with an even index belongs to $V_1$ and a vertex on $W$ with an odd index belongs to $V_2$.
  Since $D$ is a bipartite tournament, there is an arc between $u_i$ and $u_{l-i-1}$ for each integer $0 \le i \le l/2 - 1$.
  Then $u_{{l/2 }-1} \to u_{l/2}$.
Since $u_{l-1}\to u_0$, there is the smallest index $i$ among $1, \ldots, l/2 -1$ such that $u_{i} \to u_{l-i-1}$.
  Then $u_{l-i} \to u_{i-1}$ and $C := u_{i-1} \to u_i \to u_{l-i-1} \to u_{l-i} \to u_{i-1}$ is a closed directed walk of length $4$ in $D$.
  Since $D$ has neither loops nor multiple arcs, all the vertices on $C$ are distinct and so $C$ is a directed cycle in $D$.
  Since $D$ is a bipartite tournament, there is no directed cycle of odd length in $D$ and so $\kappa(D)$ is $2$ or $4$.
\end{proof}

The following is an immediate consequence of the above proposition.
\begin{Cor}\label{cor:kappa}
Let $D$ be a strongly connected $k$-partite tournament with $k$-partition $(V_1,V_2,\ldots, V_k)$ for an integer $k \ge 2$.
Then the following are true:
\begin{itemize}
  \item[(i)] if $\kappa(D)=1$, then $k \ge 3$;
  \item[(ii)] if $\kappa(D)=2$, then $k=2$ and the sets of imprimitivity of $D$ are $V_1$ and $V_2$;
  \item[(iii)] if $\kappa(D)=3$, then $k=3$ and the sets of imprimitivity of $D$ are $V_1$, $V_2$, and $V_3$;
  \item[(iv)] if $\kappa(D)=4$, then $k=2$ and there exists a partition $\{X_i, Y_i\}$ of $V_i$ for each $i=1,2$ such that $X_1$, $Y_1$, $X_2$, and $Y_2$ are the sets of imprimitivity of $D$.
\end{itemize}
\end{Cor}

The following theorem describes the limit of $\{C^m(D)\}_{m=1}^{\infty}$ for a strongly connected digraph $D$.
Given a graph $G$, we denote the clique of $G$ with the vertex set $Z$ by $K[Z]$.

\begin{Thm}[\cite{park2013matrix}]\label{thm:strong}
  If $D$ is nontrivial and strongly connected, then the limit of $\{C^m(D)\}_{m=1}^{\infty}$ is \[\bigcup_{i=1}^{\kappa(D)}{K{[U_{i}]}}\]
   where $U_1, U_2, \ldots, U_{\kappa(D)}$ are the sets of imprimitivity of $D$.
\end{Thm}

By Proposition~\ref{prop:index}, $\kappa(D)\le 4$ for a strongly connected  multipartite tournament $D$ and so we have the following theorem by Theorem~\ref{thm:strong}.

\begin{Thm}\label{prop:strong}
  Let $D$ be a strongly connected multipartite tournament.
  Then the limit of $\{C^m(D)\}_{m=1}^{\infty}$ is the disjoint union of at most four complete graphs whose vertex sets are the sets of imprimitivity of $D$.
\end{Thm}


As the limit of a strongly connected multipartite tournament is taken care of, we consider the ones that are not strongly connected from now on.

Let $D$ be a multipartite tournament with ordered strong components $Q_1,Q_2,\ldots, Q_s$ and $X$ and $Y$ be partite sets of $Q_i$ and $Q_j$, respectively, for some $1 \le i, j \le s$.
We say that $X$ and $Y$ are {\it partite-related} if the partite set of $D$ including $X$ and the partite set of $D$ including $Y$ are the same.
By definition, if $i<j$ and $X$ and $Y$ are not partite-related, then $X \to Y$.
In addition, when we denote the partite sets of $D$ by $(V_1,V_2,\ldots,V_k)$ for some positive integer $k$, unless otherwise mentioned, we assume that
\begin{equation}\label{eqn:bi}
  V_i \cap V(Q_s)=
  \begin{cases}
    U_{i}^{(s)}, & \mbox{if } \kappa(Q_s)=2;\\
  U_{i}^{(s)}, & \mbox{if } \kappa(Q_s)=3;\\
    U_{i}^{(s)} \cup U_{i+2}^{(s)}, &  \mbox{if } \kappa(Q_s)=4,
  \end{cases}
\end{equation}
for each $i=1,2,3$ ($i=1,2$ if $\kappa(Q_s)=2$ or $4$).




The above observation together with Theorem~\ref{prop:strong} give rise to the following corollary.

\begin{Cor}\label{Prop:converge}
  Let $D$ be a multipartite tournament with the nontrivial last strong component $Q_s$ for some integer $s \ge 2$.
  Then the subgraph of the limit of $\{C^m(D)\}_{m=1}^{\infty}$ induced by $V(Q_s)$ is
  \[
       \bigcup_{i=1}^{\kappa(Q_s)}{K{\left[U_{i}^{(s)}\right]}}.
\]
\end{Cor}

\section{The case $\kappa(Q_s) \in \{1,2,4\}$}\label{sec:3}

\begin{Prop}\label{prop:lastprimitive}
Let $D$ be a multipartite tournament with the primitive last strong component.
Then the limit of $\{C^m(D)\}_{m=1}^{\infty}$ is a complete graph.
\end{Prop}

\begin{proof}
Let $Q_s$ be the last strong component of $D$.
Take two vertices $u \in V(D)$ and $v \in V(Q_s)$.
  We let $N= \exp(Q_s) +1$ and take an integer $m \ge N$.
  Since $Q_s$ is primitive, $Q_s$ is nontrivial and so has at least two partite sets.
  Therefore there is an arc $(u,w)$ for some vertex $w$ in $Q_s$.
  Moreover, $Q_s$ being primitive implies that there are $(w,v)$-directed walk of length $m-1$ in $Q_s$.
  Since $(u,w)$ is an arc in $D$, there is a $(u,v)$-directed walk of length $m$ and so $v$ is an $m$-step prey of $u$.
  Since $u$ was arbitrarily chosen from $V(D)$, $v$ is an $m$-step prey of any vertex in $V(D)$.
  Thus the limit of $\{C^m(D)\}_{m=1}^{\infty}$ is a complete graph by ($\ast$).
\end{proof}

Bondy~\cite{bondy1976diconnected} showed that, for an integer $k\ge3$, a strongly connected $k$-partite tournament contains a directed cycle of length $\ell$ for each integer $3 \le \ell \le k$.
If the last strong component of a multipartite tournament $D$ has at least four partite sets, then it is primitive and so the limit of $\{C^m(D)\}_{m=1}^{\infty}$ is a complete graph by Proposition~\ref{prop:lastprimitive}.
Thus, from now on, we only consider a multipartite tournament whose last strong component has at most three partite sets.

Let $D$ be a bipartite tournament with the nontrivial last strong component $Q_s$ for some integer $s \ge 2$.
  Then $\kappa(Q_s) \in \{2,4\}$ by Proposition~\ref{prop:index} and this justifies why Theorem~\ref{thm:bipartite}
only considers the case $\kappa(Q_s)=2$ and the case $\kappa(Q_s)=4$.

\begin{Thm}\label{thm:bipartite}
  Let $D$ be a bipartite tournament with the nontrivial last strong component $Q_s$ for some integer $s \ge 2$.
  Then the limit $G$ of the graph sequence $\{C^m(D)\}_{m=1}^{\infty}$ is as follows:
\[
G \cong \begin{cases}
          G_1, & \mbox{if } \kappa(Q_s)=2; \\
          G_2, & \mbox{if } \kappa(Q_s)=4,
        \end{cases}
\]
where $G_1$ and $G_2$ are the graphs given in Figure~\ref{fig:graphs}.
\end{Thm}

\begin{proof}
Let $(V_1,V_2)$ be a bipartition of $D$.
  Since $Q_s$ is nontrivial, $\kappa(Q_s) \in \{2,4\}$ as we observed prior to the theorem statement.
  Moreover, $V_1 \cap V(Q_s) \neq \emptyset$ and $V_2 \cap V(Q_s) \neq \emptyset$.
  Then, by our assumption \eqref{eqn:bi},
\[
  V_i \cap V(Q_s)=
  \begin{cases}
    U_{i}^{(s)}, & \mbox{if } \kappa(Q_s)=2;  \\
    U_{i}^{(s)} \cup U_{i+2}^{(s)}, &  \mbox{if } \kappa(Q_s)=4,
  \end{cases}
\]
for each $i=1,2$.
Then there is an arc from each vertex in $V_1 \cap V(Q_s)$ (resp.\ $V_2 \cap V(Q_s)$) to some vertex in $V_2 \cap V(Q_s)$ (resp.\ $V_1 \cap V(Q_s)$).
  By the way, $V_1 \cap \{V(D)- V(Q_s)\} \to V_2 \cap V(Q_s)$ and
  $V_2 \cap \{V(D)- V(Q_s)\} \to V_1 \cap V(Q_s)$.
  Therefore the following are true:
  \begin{itemize}
    \item[(i)] any pair of vertices in $V_1 \cap \{V(D)- V(Q_s)\}$ has a common prey in $D$;
    \item[(ii)] any pair of vertices in $V_2 \cap \{V(D)- V(Q_s)\}$ has a common prey in $D$;
    \item[(iii)] each vertex in $V_1 \cap V(Q_s)$ and each vertex in $V_1 \cap \{V(D)- V(Q_s)\}$ have a common prey in $D$;
    \item[(iv)] each vertex in $V_2 \cap V(Q_s)$ and each vertex in $V_2 \cap \{V(D)- V(Q_s)\}$ have a common prey in $D$.
  \end{itemize}
  Therefore, by ($\ast$), we have shown that, for each $i=1,2$,
   \begin{itemize}
    \item $V_i \cap \{V(D)- V(Q_s)\}$ forms a clique in $G$;
    \item each vertex in $V_i \cap V(Q_s)$ and each vertex in $V_i \cap \{V(D)- V(Q_s)\}$ are adjacent in $G$.
  \end{itemize}
  Since $D$ is a bipartite tournament, there is no edge between a vertex in $V_1$ and a vertex in $V_2$ in $G$.
  Thus, by Corollary~\ref{Prop:converge},
   \begin{itemize}
     \item if $\kappa(Q_s)=2$, then $G \cong G_1$ where $K^{(1)}$ does not exist, $K^{(2)}=K[V_1]$, and $K^{(3)}=K[V_2]$;
     \item if $\kappa(Q_s)=4$, then $G \cong G_2$ where $K^{(1)}$ does not exist,
  $K^{(2)}=K{[V_1 - V(Q_s)]}$, $K^{(3)}=K{[V_2 - V(Q_s)]}$, $K^{(4)}=K{\left[U_{1}^{(s)}\right]}$, $K^{(5)}=K{\left[U_{3}^{(s)}\right]}$, $K^{(6)}=K{\left[U_{2}^{(s)}\right]}$, and $K^{(7)}=K{\left[U_{4}^{(s)}\right]}$.\qedhere
   \end{itemize}
\end{proof}

  The following proposition gives a detailed description of the limit of $\{C^m(D)\}_{m=1}^{\infty}$ in case where $D$ contains a strong component one of whose partite sets is not partite-related to any partite set of the last strong component.

\begin{Prop}\label{prop:not partite-related}
  For some integers $k \ge 3$ and  $s \ge 2$, let $D$ be a $k$-partite tournament with ordered strong components $Q_1,Q_2,\ldots, Q_s$.
Suppose that $Q_s$ is nontrivial and there is an integer $t \in \{1,\ldots, s-1\}$ such that a partite set of $Q_t$ is not partite-related to any partite set of $Q_s$.
Then any vertex in $\bigcup_{i=1}^{t}{V(Q_i)}$ and any vertex in $D$ are adjacent in the limit of the graph sequence $\{C^m(D)\}_{m=1}^{\infty}$.
\end{Prop}

\begin{proof}
By the hypothesis, there is a partite set of $Q_t$ that is not partite-related to any partite set of $Q_s$.
Let $Z$ be the partite set of $D$ that contains such a partite set of $Q_t$.
  Take two vertices $u \in \bigcup_{i=1}^{t}{V(Q_i)}$ and $v \in V(D)$.
  Since $Q_s$ is nontrivial, $Q_s$ has at least two partite sets and so $v \to w$ for some vertex $w$ in $Q_s$.
If $u \in Z$, then $u \to V(Q_s)$ and so $w$ is a common prey of $u$ and $v$.

Suppose $u \not\in Z$.
Take a vertex $z \in Z \cap V(Q_t)$.
  Then there is a $(u,z)$-directed path $P$ in $D$, which is true by the strong connectedness of $Q_t$ if $u \in V(Q_t)$ and by the fact that $u$ belongs to a partite set different from $Z$ (which implies $u \to Z$) if $u \in \bigcup_{i=1}^{t-1}{V(Q_i)}$.
  Let $\ell$ be the length of $P$.
  Since $Q_s$ is strongly connected, there is a $(w,x)$-directed walk $P'$ of length $\ell$ for some vertex $x \in V(Q_s)$.
  By the choice of $Z$, $z \to x$.
  Thus $P \to x$ and $v \to P'$ are a $(u,x)$-directed walk and a $(v,x)$-directed walk, respectively, of the same length $\ell+1$ in $D$.
  Therefore $x$ is an $(\ell+1)$-step common prey of $u$ and $v$ in $D$.
  Therefore $u$ and $v$ have an $m$-step common prey in $D$ for some positive integer $m$ and so the statement is true by ($\ast$).
\end{proof}

In the rest of this paper, we will frequently use the notation $\bigcup_{i=p}^{q}{V(Q_i)}$ for positive integers $p$, $q$.
We will regard $\bigcup_{i=p}^{q}{V(Q_i)}$ as an empty set if $p>q$.

\begin{Thm}\label{Thm:tritrinontri}
  Let $D$ be a $k$-partite tournament with $k$-partition $(V_1,V_2,\ldots,V_k)$ and ordered strong components $Q_1,Q_2,\ldots, Q_s$ for some integers $k \ge 3$ and  $s \ge 2$.
Suppose that $Q_s$ is nontrivial and $\kappa(Q_s)\in \{2,4\}$.
Then the limit $G$ of the graph sequence $\{C^m(D)\}_{m=1}^{\infty}$ is as follows:
\[
G \cong \begin{cases}
          G_1, & \mbox{if } \kappa(Q_s)=2; \\
          G_2, & \mbox{if } \kappa(Q_s)=4,
        \end{cases}
\]
where $G_1$ and $G_2$ are the graphs given in Figure~\ref{fig:graphs}.
\end{Thm}

\begin{proof}
Since $\kappa(Q_s) \in \{2,4\}$, $Q_s$ is a bipartite tournament by Corollary~\ref{cor:kappa}
and so $V(Q_s) \subseteq V_1 \cup V_2$ by \eqref{eqn:bi}.
 Since $k \ge 3$, there exists an integer $t \in \{1,\ldots,s-1\}$ such that there is a partite set of $Q_t$ which is not partite-related to any partite set of $Q_s$.
 We may assume that $t$ is the largest among such integers.
 Then
  \[
  V(D) - \bigcup_{i=1}^{t}{V(Q_i)} \subseteq V_1 \cup V_2.
  \]
  Therefore the subdigraph $D'$ induced by $\bigcup_{i=t+1}^{s}V(Q_i)$ is a bipartite tournament with partite sets $A_1:=V_1 \cap \bigcup_{i=t+1}^{s}V(Q_i)$ and $A_2:=V_2 \cap \bigcup_{i=t+1}^{s}V(Q_i)$.
  In the proof of Theorem~\ref{thm:bipartite}, it is shown that the subgraph $G_1$ of $G$ induced by $V(D')$ contains $H_1:=K{[A_1]} \cup K{[A_2]}$ if $\kappa(Q_s)=2$ and $H_2:=\bigcup_{i=1}^{2}\left( K{[A_i - V(Q_s)]} \vee \left(K{\left[U_{i}^{(s)}\right]} \cup K{\left[U_{i+2}^{(s)}\right]}  \right)  \right)$ if $\kappa(Q_s)=4$.
  Since $Q_{t+1},\ldots,Q_s$ are the last components among the ordered strong components, $G_1$ cannot have edges other than ones in $H_1$ or $H_2$ and so $G_1 = H_1$ or $G_1 = H_2$.
 By Proposition~\ref{prop:not partite-related}, $\bigcup_{i=1}^{t}{V(Q_i)}$ forms a clique in $G$ and any vertex in $\bigcup_{i=1}^{t}V(Q_i)$ and any vertex in $\bigcup_{i=t+1}^{s}V(Q_i)$ are adjacent in $G$.
   Therefore we may conclude that
   \begin{itemize}
     \item if $\kappa(Q_s)=2$, then $G \cong G_1$ where $K^{(1)}=K\left[\bigcup_{i=1}^{t}V(Q_i)\right]$, $K^{(2)}=K[A_1]$, and $K^{(3)}=K[A_2]$;
     \item if $\kappa(Q_s)=4$, then $G \cong G_2$ where $K^{(1)}=K\left[\bigcup_{i=1}^{t}V(Q_i)\right]$, $K^{(2)}=K[A_1 - V(Q_s)]$, $K^{(3)}=K[A_2 - V(Q_s)]$, $K^{(4)}=K\left[U_{1}^{(s)}\right]$, $K^{(5)}=K\left[U_{3}^{(s)}\right]$, $K^{(6)}=K\left[U_{2}^{(s)}\right]$, and $K^{(7)}=K\left[U_{4}^{(s)}\right]$.\qedhere
   \end{itemize}
\end{proof}

%

\section{The case $\kappa(Q_s)=3$}\label{sec:4}
By Theorems~\ref{prop:strong}, \ref{thm:bipartite}, and \ref{Thm:tritrinontri}, it remains to characterize the limit of the graph sequence $\{C^m(D)\}_{m=1}^{\infty}$ for a multipartite tournament $D$ with $\kappa(Q_s) = 3$ where $Q_1, Q_2, \ldots, Q_s$ are ordered strong components of $D$ for some integer $s \ge 2$.
First, we need the following lemmas.

\begin{Lem}\label{Lem:useful}
  Let $D$ be a multipartite tournament with ordered strong components $Q_1$ and $Q_2$.
  Suppose that, for a positive integer $L$ and some vertices $x \in V(Q_1)$ and $y \in V(Q_2)$, there exists an $(x,y)$-directed walk of length $\ell$ for each $\ell \ge L$.
  Then there exists a positive integer $L'$ such that, for any vertex $u$ in $Q_1$ and for any vertex $v$ in $Q_2$, there exists a $(u,v)$-directed walk of length $\ell$ for each $\ell \ge L'$.
\end{Lem}

\begin{proof}
  Since $Q_1$ and $Q_2$ are strongly connected, there exist a $(u,x)$-directed path in $Q_1$ and a $(y,v)$-directed path in $Q_2$ for any $u \in V(Q_1)$ and $v \in V(Q_2)$.
  We let
  \[
  L'=\max_{u \in V(Q_1)}{d(u,x)}+L + \max_{v \in V(Q_2)}{d(y,v)}\]
  ($d(w,z)$ stands for the length of a shortest $(w,z)$-directed path), which is desired.
\end{proof}

Let $D$ be a multipartite tournament with ordered strong components $Q_1$ and $Q_2$.
We say that $D$ is {\it unusual} if $D$ satisfies the following conditions:
\begin{itemize}
  \item $\kappa(Q_1)=\kappa(Q_2)=3$;
  \item there exists an integer $j \in \{0,1,2\}$ such that $U_i^{(1)}$ and $U_{i+j}^{(2)}$ are partite related for each $i \in \{1,2,3\}$ (identify $U_4^{(2)}$ and $U_5^{(2)}$ with $U_1^{(2)}$ and $U_2^{(2)}$, respectively).
\end{itemize}

\begin{Lem}\label{lem:clique}
  Suppose that a multipartite tournament $D$ has ordered strong components $Q_1,Q_2,\ldots, Q_s$ and $\kappa(Q_s)=3$ for some integer $s \ge 2$.
  Then, for the limit $G$ of the graph sequence $\{C^m(D)\}_{m=1}^{\infty}$, the following are true:
  \begin{itemize}
    \item[(i)] $V(D) - V(Q_s)$ forms a clique in $G$;
     \item[(ii)] if $\kappa(Q_{s-1}) \neq 3$ and $Q_{s-1}$ is nontrivial, then each vertex in $V(D) - V(Q_s)$ and each vertex in $V(Q_s)$ are adjacent in $G$;
    \item[(iii)] if $\kappa(Q_{s-1})=3$, then each vertex in $V(D) - \left(V(Q_{s-1}) \cup V(Q_s)\right)$ and each vertex in $V(Q_s)$ are adjacent in $G$ and, further, in case where $V(Q_{s-1}) \cup V(Q_s)$ does not induce an unusual digraph, each vertex in $V(Q_{s-1})$ and each vertex in $V(Q_s)$ are adjacent in $G$.
  \end{itemize}

\end{Lem}

\begin{proof}
 Let $D$ be a $k$-partite tournament with $k$-partition $(V_1,V_2,\ldots,V_k)$ for an integer $k \ge 2$.
   To show part (i), take two vertices $u_1$ and $u_2$ in $V(D) - V(Q_s)$.
  Then, since $\kappa(Q_s)=3$, $V_j$ does not contain any of $u_1$ and $u_2$ for some $j \in \{1,2,3\}$.
  By \eqref{eqn:bi}, $U_{j}^{(s)} \subseteq V_j$.
  Therefore each vertex in $U_{j}^{(s)}$ is a common prey of $u_1$ and $u_2$.
  Hence $u_1$ and $u_2$ are adjacent in $G$ by ($\ast$).
  Since $u_1$ and $u_2$ were arbitrarily chosen, (i) is true.

  To show parts (ii) and (iii), take $u_1 \in V(Q_s)$ and
  assume
  \[u_1 \in U_{i}^{(s)}
  \]
  for some $i \in \{1,2,3\}$.

 To show (ii),  suppose that $\kappa(Q_{s-1}) \neq 3$ and $Q_{s-1}$ is nontrivial, and take  $u_2 \in V(D) - V(Q_s)$.
  Then, by Corollary~\ref{cor:kappa}, $Q_{s-1}$ is either primitive or bipartite, so $Q_{s-1}$ has at least two partite sets.
   Therefore there exists an arc from $u_2$ to a vertex in $Q_{s-1}$.
We take an out-neighbor of $u_2$ that belongs to $Q_{s-1}$ and denote it by $\varphi_{u_2}$.

   We first consider the case in which $Q_{s-1}$ is primitive.
  We let $N= \exp(Q_{s-1}) +2$ and take an integer $m \ge N$.
 Since $Q_s$ is strongly connected, we may take a vertex $v$ in $Q_s$ such that there is a $(u_1,v)$-directed walk of length $m$.
  Since $Q_{s-1}$ has at least two partite sets, there is an in-neighbor $x \in V(Q_{s-1})$ of $v$.
  By the choice of $N$ and the case assumption that $Q_{s-1}$ is primitive, there is a $(\varphi_{u_2},x)$-directed walk $W$ of length $m-2$.
  Then $u_2  \to W  \to v$ is a $(u_2,v)$-directed walk of length $m$.
  Therefore $v$ is an $m$-step common prey of $u_1$ and $u_2$.
  Hence $u_1$ and $u_2$ are adjacent in $G$ by ($\ast$).

Now we consider the case in which $Q_{s-1}$ is bipartite.
Then, for each nonnegative integer $a$, there exists a vertex $y_a \in V(Q_{s-1})$
such that there is a $(\varphi_{u_2},y_a)$-directed walk of length $2a$.
Since $Q_{s-1}$ is bipartite, $\varphi_{u_2}$ and $y_a$ belong to the same partite set, say $Y$, in $Q_{s-1}$.
Since $\kappa(Q_s)=3$, $Q_s$ has three partite sets by Corollary~\ref{cor:kappa}.
Then, since $y_a$ belongs to $Y$ for any nonnegative integer $a$, there exists a vertex $z \in V(Q_s)$ such that $(y_a,z)$ is an arc for any nonnegative integer $a$.
Take a vertex $w \in U_{i}^{(s)}$.
 For some $\alpha \in \{1,2,3\}$ and each nonnegative integer $b$, there exists a $(z,w)$-directed walk of length $3b+\alpha$.
Hence there exists a $(u_2,w)$-directed walk of length $2a+3b+2+\alpha$ for any nonnegative integers $a$ and $b$.
Since the Frobenius number of the set $\{2,3\}$ is $1$, it is guaranteed that there is a $(u_2,w)$-directed walk of length $\ell$ for any $\ell \ge 4+\alpha$ (recall that the Frobenius number of the set $\{p,q\}$ for relatively prime $p$ and $q$ is the largest integer $b$ such that $b$ cannot be written as $ps+qt$ for any nonnegative integers $s$ and $t$ and is known to be $pq-p-q$).
Since  $u_1, w\in U_{i}^{(s)}$ and $\kappa(Q_s)=3$, there exists a $(u_1,w)$-directed walk of length $3c$ for each positive integer $c$.
Therefore we have shown that there exist a $(u_1,w)$-directed walk and a $(u_2,w)$-directed walk of the same length.
  Hence $u_1$ and $u_2$ are adjacent in $G$ by ($\ast$).
  Since $u_1$ and $u_2$ were arbitrarily chosen and are adjacent in $G$ in each case, (ii) is true.

  To show (iii), suppose $\kappa(Q_{s-1})=3$.
  Then $Q_{s-1}$ has three partite sets by Corollary~\ref{cor:kappa}, so there exists a partite set $X$ of $Q_{s-1}$ such that $X \not\subseteq X_{i+1} \cup X_{i+2}$ where $X_j$ is the partite set of $D$ including $U_j^{(s)}$ for each $j \in \{1,2,3\}$ (recall $u_1 \in U_{i}^{(s)}$ and note that $X_j=V_j$ for $j=1,2,3$).
  For a nonnegative integer $a$ and for each $j=1,2,3$, we identify $U^{(s)}_{3a+j}$ with $U^{(s)}_{j}$ and $X_{3a+j}$ with $X_{j}$.

Take $u_2 \in V(D) - \left(V(Q_{s-1}) \cup V(Q_s)\right)$.
  If $u_2 \not\in X_{i+1}$, then each vertex in $U_{i+1}^{(s)}$ is a common prey of $u_1$ and $u_2$ and so, by ($\ast$), $u_1$ and $u_2$ are adjacent in $G$.
  Suppose $u_2 \in X_{i+1}$.
  Then there is an arc from $u_2$ to each vertex in $X$.
  Since $X \not\subseteq X_{i+1} \cup X_{i+2}$, there is an arc from any vertex in $X$ to any vertex in  $U_{i+2}^{(s)}$, and so each vertex in $U_{i+2}^{(s)}$ is a $2$-step common prey of $u_1$ and $u_2$.
  Thus $u_1$ and $u_2$ are adjacent in $G$ by ($\ast$).
  Since $u_1$ and $u_2$ were arbitrarily chosen, each vertex in $V(D) - \left(V(Q_{s-1}) \cup V(Q_s)\right)$ and each vertex in $V(Q_s)$ are adjacent in $G$.

  Now, to show the ``further'' part, suppose that the digraph induced by $V(Q_{s-1}) \cup V(Q_s)$ is not unusual.
  We will claim that the digraph induced by $V(Q_{s-1}) \cup V(Q_s)$ satisfies the hypothesis of Lemma~\ref{Lem:useful} in each of the following cases.

{\it Case 1.}
Suppose that there exists a partite set $Y$ of $Q_s$ which is not partite-related to any partite set of $Q_{s-1}$.
Take a positive integer $\ell$, a vertex $x \in V(Q_{s-1})$, and a vertex $y \in Y$.
Since $Q_{s-1}$ is a nontrivial strong component, there exists an $(x,z)$-directed walk of length $\ell-1$ in $Q_{s-1}$ for some vertex $z \in V(Q_{s-1})$.
Since any partite set of $Q_{s-1}$ and $Y$ are not partite-related,
there exists an arc from $z$ to $y$, which results in an $(x,y)$-directed walk of length $\ell$.

{\it Case 2.}
Suppose that any partite set of $Q_s$ is included in the partite set of $D$ including some partite set of $Q_{s-1}$.
Then $V(Q_{s-1}) \subseteq V_1 \cup V_2 \cup V_3$.
Since the digraph induced by $V(Q_{s-1}) \cup V(Q_s)$ is not unusual, without loss of generality, we may assume that  $U_{1}^{(s-1)} \subseteq V_1$, $U_{2}^{(s-1)} \subseteq V_3$, and $U_{3}^{(s-1)} \subseteq V_2$.
Take $x \in U_{1}^{(s-1)}$ and $y \in U_{2}^{(s)}$.
Then $(x,y) \in A(D)$.
Since $\kappa(Q_{s-1})=3$, there exists an $(x,x)$-directed walk $P_{a}$ of length $3a$ in $Q_{s-1}$ for each nonnegative integer $a$.
Take vertices $z_1 \in U_{2}^{(s-1)}$ and $z_2 \in U_{1}^{(s)}$.
Then $(x,z_1)$, $(z_1,y)$, $(z_1,z_2), (z_2,y) \in A(D)$.
Therefore $P_a \to y$, $P_a \to z_1 \to y$, and $P_a \to z_1 \to z_2 \to y$ are $(x,y)$-directed walks of length $3a+1$, $3a+2$, and $3a+3$, respectively, for each nonnegative integer $a$.
Hence there is a directed walk of length $\ell$ from $x$ to $y$ for each positive integer $\ell$.
Consequently, we have shown that the digraph induced by $V(Q_{s-1}) \cup V(Q_s)$ satisfies the hypothesis of Lemma~\ref{Lem:useful} in each case.
Thus, there exists a positive integer $L$ such that for any vertex $x$ in $Q_{s-1}$ and for any vertex $y$ in $Q_s$, there exists an $(x,y)$-directed walk of length $\ell$ for each $\ell \ge L$.

Take $u \in V(Q_{s-1})$ and $v \in V(Q_s)$.
Since $Q_s$ is strong, there exists a $(v,w)$-directed walk of length $L$ for some vertex $w \in V(Q_s)$.
By the previous claim, there exists a $(u,w)$-directed walk of length $L$.
Hence $u$ and $v$ are adjacent in $G$ by ($\ast$).
Since $u$ and $v$ were arbitrarily chosen, we have shown the ``further'' part of (iii).
\end{proof}

\begin{Lem}\label{claim:a}
 Let $D$ be a $k$-partite tournament having $k$-partition $(V_1,V_2, \ldots, V_k)$ and ordered strong components $Q_1,Q_2,\ldots, Q_s$ with $\kappa(Q_s)=3$ for some integers $k,s \ge 2$ and let $G$ be the limit of the graph sequence $\{C^m(D)\}_{m=1}^{\infty}$.
  Further, let $D_1$ be the subdigraph induced by $\bigcup_{i=p}^{q}V(Q_i) \cap V_j$ for some $p$, $q$ with $1 \le p \le q < s$ and some $j \in \{1,2,3\}$.
\begin{itemize}
  \item[(a)] every vertex in $D_1$ and every vertex in $U_{j}^{(s)} \cup U_{j+1}^{(s)}$ are adjacent in $G$ (identify $U_{i}^{(s)}$ with $U_{r}^{(s)}$ where $r \in \{1,2,3\}$ and $r \equiv i \pmod 3$);
  \item[(b)] Suppose that $\bigcup_{i=p}^{q}V(Q_i) \subseteq V_j$ and $\bigcup_{i=q+1}^{s-1}V(Q_i) \subseteq V_{j+1}$ (identify $V_4$ with $V_1$). Then, in $G$, there is no edge between a vertex in $\bigcup_{i=p}^{q}{V(Q_{i})}$ and a vertex in $U_{j+2}^{(s)}$.
\end{itemize}
\end{Lem}

\begin{proof}
To show (a), take a vertex $u$ in $D_1$ and a vertex $v$ in $U_{j}^{(s)} \cup U_{j+1}^{(s)}$.
  If $v \in U_{j}^{(s)}$ (resp.\ $U_{j+1}^{(s)}$), then each vertex in $U_{j+1}^{(s)} \subseteq V_{j+1}$ (resp.\ $U_{j+2}^{(s)} \subseteq V_{j+2}$) is a common prey of $u$ and $v$ (note that, by \eqref{eqn:bi}, $U_{i}^{(s)} \subseteq V_{i}$ for each $i =1,2,3$).
  Therefore $u$ and $v$ are adjacent in $G$ by ($\ast$).

  To show (b), suppose that $\bigcup_{i=p}^{q}V(Q_i) \subseteq V_j$ and $\bigcup_{i=q+1}^{s-1}V(Q_i) \subseteq V_{j+1}$.
  To the contrary, assume that a vertex $u \in \bigcup_{i=p}^{q}{V(Q_{i})} \subseteq V_j$ and a vertex $v \in U_{j+2}^{(s)}$ have an $m$-step common prey $z$ for some positive integer $m$.
  We note that $v$ belongs to the last strong component $Q_s$ and $U_{1}^{(s)},U_{2}^{(s)}$, and $U_{3}^{(s)}$ are the sets of imprimitivity of $Q_s$.
  Therefore
  \begin{equation}\label{eqn:zz}
    z \in U_{j+2+m}^{(s)}
  \end{equation}
     and there is an $(u,z)$-directed walk $W$ of length $m$ in $D$.
  Let $w_1$ and $w_2$ be the vertices located right next to $u$ and $w_1$, respectively, on $W$.
  Since $\bigcup_{i=p}^{q}{V(Q_{i})}  \subseteq V_j$, there is no arc between any two vertices in $\bigcup_{i=p}^{q}{V(Q_{i})}$ and so $w_1 \in \bigcup_{i=q+1}^{s}{V(Q_{i})}$.
  If $w_1 \in V(Q_s)$, then $w_1 \in  U_{j+1}^{(s)} \cup U_{j+2}^{(s)}$ since $u \in V_j$ (recall $\kappa(Q_s)=3$).
  Now consider the case $w_1 \in \bigcup_{i=q+1}^{s-1}{V(Q_{i})} \subseteq V_{j+1}$.
   Then $w_2 \in V(Q_s)$ since there is no arc between any two vertices in $\bigcup_{i=q+1}^{s-1}{V(Q_{i})}$.
  Since $w_1 \in V_{j+1}$, $w_2 \in  U_{j}^{(s)} \cup U_{j+2}^{(s)}$.
  Then $z$ belongs to $U_{j+m}^{(s)}$ or $U_{j+1+m}^{(s)}$ in each case, which contradicts \eqref{eqn:zz}.
 Thus there is no edge between a vertex in $\bigcup_{i=p}^{q}{V(Q_{i})}$ and a vertex in $U_{j+2}^{(s)}$ in $G$.
\end{proof}

\begin{Thm}\label{thm:s-1trivial}
  Suppose that a multipartite tournament $D$ has ordered strong components $Q_1,Q_2,\ldots, Q_s$ with trivial $Q_{s-1}$ and $\kappa(Q_s)=3$ for some integer $s \ge 2$.
  Then the limit of the graph sequence $\{C^m(D)\}_{m=1}^{\infty}$ is isomorphic to $G_3$ given in Figure~\ref{fig:graphs}.
\end{Thm}

\begin{proof}
Let $G$ be the limit of the graph sequence $\{C^m(D)\}_{m=1}^{\infty}$.
By Lemma~\ref{lem:clique}(i) and Corollary~\ref{Prop:converge},
\begin{itemize}
  \item[(a)] $V(D) - V(Q_s)$ forms a clique in $G$ and $V(Q_s)$ induces $\bigcup_{i=1}^{3}{K{\left[U_{i}^{(s)}\right]}}$ in $G$.
\end{itemize}

  Since $Q_{s-1}$ is trivial, the digraph induced by $V(Q_{s-1}) \cup V(Q_s)$ is not unusual.
Let $V(Q_{s-1}) = \{v_{s-1}\}$.
If there is an integer $t \in \{1,\ldots, s-1\}$ such that a partite set of $Q_t$ is not partite-related to any partite set of $Q_s$, then we take $t$ as the largest among such integers.
If such a $t$ does not exist, that is, $D$ is a tripartite tournament, then we let $t=0$.
If $t \ge 1$, then, by Proposition~\ref{prop:not partite-related},
\begin{itemize}
  \item[(b)] $\bigcup_{i=1}^{t}{V(Q_i)}$ forms a clique in $G$ and any vertex in $\bigcup_{i=1}^{t}V(Q_i)$ and any vertex in $\bigcup_{i=t+1}^{s}V(Q_i)$ are adjacent in $G$.
\end{itemize}
Especially, if $t=s-1$, then $G \cong G_3$ where $K^{(2)}$, $K^{(3)}$, and $K^{(4)}$ do not exist and
\[
K^{(1)}=K\left[\bigcup_{i=1}^{s-1}V(Q_i)\right], \quad K^{(5)}=K\left[U_{1}^{(s)}\right], \quad K^{(6)}=K\left[U_{2}^{(s)}\right], \quad \mbox{and} \quad K^{(7)}=K\left[U_{3}^{(s)}\right].
\]
Now assume $t < s-1$.
Then every partite set of $Q_{i}$ is partite-related to some partite set of $Q_s$ for each $i= t+1, \ldots, s-1$.
Recall that, by \eqref{eqn:bi},  $V_i \cap V(Q_s)=  U_{i}^{(s)}$ for each $i=1,2,3$ where $(V_1,V_2,\ldots,V_k)$ is a $k$-partition of $D$.
Without loss of generality, we may assume that
\begin{equation}\label{eqn:s-1}
v_{s-1} \in V_2.
\end{equation}
Suppose that $\bigcup_{i=t+1}^{s-1}V(Q_i) \subseteq V_j$ for some $j \in \{1,2,3\}$.
Then $j=2$ and, by Lemma~\ref{claim:a},
\begin{itemize}
  \item every vertex in $\bigcup_{i=t+1}^{s-1}V(Q_i)$ and every vertex in $U_{2}^{(s)} \cup U_{3}^{(s)}$ are adjacent in $G$;
  \item there is no edge between a vertex in $\bigcup_{i=t+1}^{s-1}{V(Q_{i})}$ and a vertex in $U_{1}^{(s)}$ in $G$.
\end{itemize}
Therefore, by (a) and (b), $G \cong G_3$ where $K^{(2)}$ and $K^{(3)}$ do not exist and
\[
K^{(1)}=K\left[\bigcup_{i=1}^{t}V(Q_i)\right], \quad
K^{(4)}=K\left[\bigcup_{i=t+1}^{s-1}V(Q_i)\right],
\]
\[
K^{(5)}=K\left[U_{1}^{(s)}\right], \quad K^{(6)}=K\left[U_{2}^{(s)}\right], \quad \mbox{and} \quad K^{(7)}=K\left[U_{3}^{(s)}\right].
\]

Now consider the case in which $\bigcup_{i=t+1}^{s-1}V(Q_i) \not\subseteq V_i$ for any $i \in \{1,2,3\}$.
Then there is an integer $t_1$ such that $t < t_1 \le s-1$ and $V(Q_{t_1}) \not \subseteq V_2$.
We may assume that $t_1$ is the largest integer among such integers.
By \eqref{eqn:s-1}, $t_1 \neq s-1$ and $\bigcup_{i=t_1+1}^{s-1}{V(Q_{i})\subseteq V_2}$.
Then, by Lemma~\ref{claim:a},
\begin{itemize}
  \item[(c1)] every vertex in $\bigcup_{i=t_1+1}^{s-1}V(Q_i)$ and every vertex in $U_{2}^{(s)} \cup U_{3}^{(s)}$ are adjacent in $G$;
  \item[(d1)] there is no edge between a vertex in $\bigcup_{i=t_1+1}^{s-1}{V(Q_{i})}$ and a vertex in $U_{1}^{(s)}$ in $G$.
\end{itemize}

It remains to check the adjacency of a vertex belonging to $\bigcup_{i=t+1}^{t_1}{V(Q_{i})}$ and a vertex belonging to $V(Q_s)$ in $G$.

{\it Case 1.} $Q_{t_1}$ is nontrivial.
Suppose that the digraph induced by $V(Q_{t_1}) \cup V(Q_s)$ is not unusual.
Then any vertex in $\bigcup_{i=t+1}^{t_1}{V(Q_{i})}$ and any vertex in $V(Q_s)$ are adjacent in $G$ by Lemma~\ref{lem:clique}(ii) and (iii) applied to the subdigraph induced by $\bigcup_{i=t+1}^{t_1}{V(Q_{i})} \cup V(Q_s)$, which is a multipartite tournament with ordered strong components $Q_{t+1},\ldots,Q_{t_1},Q_{s}$.
Now suppose that the digraph induced by $V(Q_{t_1}) \cup V(Q_s)$ is unusual.
Then
\[\kappa(Q_{t_1})=3\]
 and so, if $t+1 \le t_1-1$, each vertex in $\bigcup_{i=t+1}^{t_1-1}{V(Q_{i})}$ and each vertex in $V(Q_s)$ are adjacent in $G$ by Lemma~\ref{lem:clique}(iii) applied to the subdigraph induced by $\bigcup_{i=t+1}^{t_1}{V(Q_{i})} \cup V(Q_s)$, which is a multipartite tournament with ordered strong components $Q_{t+1},\ldots,Q_{t_1},Q_{s}$.
If $t+1 = t_1$, then $\bigcup_{i=t+1}^{t_1-1}{V(Q_{i})} = \emptyset$.
Therefore, in either case, any vertex in $\bigcup_{i=t+1}^{t_1}{V(Q_{i})}$ and any vertex in $V(Q_s)$ are adjacent in $G$ as long as any vertex in $V(Q_{t_1})$ and any vertex in $V(Q_s)$ are adjacent in $G$, which is true as shown below.

 For each $i=1,2,3$, let $X_i$ be the partite set of $D$ including $U_i^{(s)}$ (note that $X_i=V_i$) and take a vertex $x_i$ in $U_{i}^{(t_1)}$ and a vertex $y_{i}$ in $U_{i}^{(s)}$.
 Now fix $i \in \{1,2,3\}$.
   Without loss of generality, we may assume $U_{i}^{(t_1)} \subseteq X_i$.
  For a nonnegative integer $a$, we identify $U^{(t_1)}_{3a+i}$ with $U^{(t_1)}_{i}$, $U^{(s)}_{3a+i}$ with $U^{(s)}_{i}$, and $X_{3a+i}$ with $X_{i}$.
Then, by Lemma~\ref{claim:a}(a) applied to the digraph induced by $V(Q_{t_1}) \cup V(Q_s)$ and the subdigraph induced by $V(Q_{t_1}) \cap V_i $, each vertex in $U_{i}^{(t_1)}$ and each vertex in $U_{j}^{(s)}$ are adjacent in $G$ if $j=i$ or $i+1$.
Accordingly, it remains to consider the case $j=i+2$.

The fact $\kappa(Q_{s})=\kappa(Q_{t_1})=3$ guarantees that there exist a $(y_{i+2},y_1)$-directed walk of length $5-i$ in $Q_{s}$ (identify $y_4$ and $y_5$ with $y_1$ and $y_2$, respectively) and an $(x_i,x_3)$-directed walk of length $3-i$ in $Q_{t_1}$.
Since $v_{s-1} \in V_2$, there are arcs $(x_3, v_{s-1})$ and $(v_{s-1}, y_1)$ in $D$.
Hence there exists an $(x_i, y_1)$-directed walk of length $5-i$ in $D$ and so $y_1$ is a $(5-i)$-step common prey of $x_i$ and $y_{i+2}$.
Since $i$, $x_i$, and $y_{i+2}$ were chosen arbitrarily, any vertex in $U_{i}^{(t_1)}$ and any vertex in $U_{i+2}^{(s)}$ are adjacent in $G$ by ($\ast$).
Therefore, by (a), (b), (c1), and (d1), $G \cong G_3$ where $K^{(2)}$ and $K^{(3)}$ do not exist and
\[
K^{(1)}=K\left[\bigcup_{i=1}^{t_1}V(Q_i)\right], \quad
K^{(4)}=K\left[\bigcup_{i=t_1+1}^{s-1}V(Q_i)\right],
\]
\[
K^{(5)}=K\left[U_{1}^{(s)}\right], \quad K^{(6)}=K\left[U_{2}^{(s)}\right], \quad \mbox{and} \quad K^{(7)}=K\left[U_{3}^{(s)}\right].
\]

{\it Case 2.} $Q_{t_1}$ is trivial.
Let $V(Q_{t_1}) = \{v_{t_1}\}$.
Take $u \in \bigcup_{i=t+1}^{t_1}{V(Q_i)}$ and $v \in V(Q_{s})$.
By the choice of $t$, $u \in X_j$ for some $j \in \{1,2,3\}$.
Then the subdigraph induced by $\{u\} \cup V(Q_s)$ is a multipartite tournament with ordered strong component $\{u\}$ and $Q_s$.
Since $\{u\} \subseteq X_j$, it follows from Lemma~\ref{claim:a}(a) that $u$ and $v$ are adjacent if $v \in X_{j} \cup X_{j+1}$.
Now we assume that $v \in X_{j+2}$.
By~\eqref{eqn:s-1},
\begin{equation}\label{eqn:from_s-1}
v_{s-1} \to U_1^{(s)} \cup U_3^{(s)}.
\end{equation}

{\it Subcase 2-1.} $v_{t_1} \in X_3$.
Then
\begin{equation}\label{eqn:from_t_1}
v_{t_1} \to v_{s-1} \quad \mbox{and} \quad v_{t_1} \to U_2^{(s)}.
\end{equation}
If $j \neq 3$, then $u \in \bigcup_{i=t+1}^{t_1-1}{V(Q_i)}$ since $v_{t_1} \in V_3$ and so
\begin{equation}\label{eqn:t_1}
u \to v_{t_1}.
\end{equation}
Suppose $j=1$.
Then $u \in X_1$ and $v \in X_3$.
By \eqref{eqn:from_t_1} and \eqref{eqn:t_1}, $u \to v_{t_1} \to U_{2}^{(s)}$.
Since $\kappa(Q_{s})=3$, we have $v \to U_{1}^{(s)} \to U_{2}^{(s)}$ and so each vertex in $U_{2}^{(s)}$ is a $2$-step common prey of $u$ and $v$.
Suppose $j=2$.
Then $u \in X_2$ and $v \in X_1$.
By \eqref{eqn:from_s-1}, \eqref{eqn:from_t_1}, and \eqref{eqn:t_1}, $u \to v_{t_1} \to v_{s-1} \to U_{1}^{(s)}$.
Since $\kappa(Q_{s})=3$, we have $v \to U_{2}^{(s)} \to U_{3}^{(s)} \to U_{1}^{(s)}$ and so each vertex in $U_{1}^{(s)}$ is a $3$-step common prey of $u$ and $v$.
Suppose $j=3$.
Then $u \in X_3$ and $v \in X_2$.
By \eqref{eqn:s-1} and \eqref{eqn:from_s-1}, $u \to v_{s-1} \to U_{1}^{(s)}$.
Since $\kappa(Q_{s})=3$, we have $v \to U_{3}^{(s)} \to U_{1}^{(s)}$ and so each vertex in $U_{1}^{(s)}$ is a $2$-step common prey of $u$ and $v$.
Consequently, $u$ and $v$ are adjacent in $G$ by ($\ast$) if $v \in X_{j+2}$.
Since $u$ and $v$ were arbitrarily chosen, we may conclude that any vertex in $\bigcup_{i=t+1}^{t_1}{V(Q_i)}$ and any vertex in $V(Q_{s})$ are adjacent in $G$.
Therefore, by (a), (b), (c1), and (d1), $G \cong G_3$ where $K^{(2)}$ and $K^{(3)}$ do not exist and
\[
K^{(1)}=K\left[\bigcup_{i=1}^{t_1}V(Q_i)\right], \quad
K^{(4)}=K\left[\bigcup_{i=t_1+1}^{s-1}V(Q_i)\right],
\]
\[
K^{(5)}=K\left[U_{1}^{(s)}\right], \quad K^{(6)}=K\left[U_{2}^{(s)}\right], \quad \mbox{and} \quad K^{(7)}=K\left[U_{3}^{(s)}\right].
\]

{\it Subcase 2-2.} $v_{t_1} \in X_1$.
Let $t_2$ be the smallest nonnegative integer less than $t_1$ such that $\bigcup_{i=t_2+1}^{t_1}{V(Q_{i})\subseteq X_1}$.
By the choice of $t$, $t \le t_2$.
Then, by Lemma~\ref{claim:a},
\begin{itemize}
  \item[(c2)] every vertex in $\bigcup_{i=t_2+1}^{t_1}V(Q_i)$ and every vertex in $U_{1}^{(s)} \cup U_{2}^{(s)}$ are adjacent in $G$.
   \item[(d2)] there is no edge between a vertex in $\bigcup_{i=t_2+1}^{t_1}{V(Q_{i})}$ and a vertex in $U_{3}^{(s)}$ in $G$.
 \end{itemize}

If $t_2=t$, by (a), (b), (c1), (d1), (c2), and (d2), $G \cong G_3$ where
\[
K^{(1)}=K\left[\bigcup_{i=1}^{t}V(Q_i)\right], \quad
K^{(3)}=K\left[\bigcup_{i=t+1}^{t_1}V(Q_i)\right], \quad
K^{(4)}=K\left[\bigcup_{i=t_1+1}^{s-1}V(Q_i)\right],
\]
\[
K^{(5)}=K\left[U_{1}^{(s)}\right], \quad K^{(6)}=K\left[U_{2}^{(s)}\right], \quad \mbox{and} \quad K^{(7)}=K\left[U_{3}^{(s)}\right].
\]

Suppose $t_2 >t$.
Take a vertex $u \in \bigcup_{i=t+1}^{t_2}{V(Q_{i})}$ and a vertex $v \in V(Q_s)$.
If $u \not \in X_1 \cup X_2 \cup X_3$, then there is an arc from $u$ to every vertex in $V(Q_s)$, and so $u$ and $v$ are adjacent in $G$ by ($\ast$).
Suppose $u  \in X_j$ for some $j \in \{1,2,3\}$.
  If $v \in U_{j}^{(s)}$ (resp.\ $U_{j+1}^{(s)}$), then each vertex in $U_{j+1}^{(s)} \subseteq X_{j+1}$ (resp.\ $U_{j+2}^{(s)} \subseteq X_{j+2}$) is a common prey of $u$ and $v$.
  Therefore $u$ and $v$ are adjacent in $G$ by ($\ast$) if $u \in X_j$ and $v \in U_{j}^{(s)} \cup U_{j+1}^{(s)}$ for some $j \in \{1,2,3\}$.

Now consider the case in which $u \in X_j$ and $v \in U_{j+2}^{(s)}$ for some $j \in \{1,2,3\}$.
Suppose $j=1$.
Then $v \to U_{1}^{(s)} \to U_{2}^{(s)} \to U_{3}^{(s)}$.
By the choice of $t_2$, $V(Q_{t_2})-X_1 \neq \emptyset$.
If $u \in \bigcup_{i=1}^{t_2-1}{V(Q_{i})}$, then $u$ has an out-neighbor in $V(Q_{t_2})-X_1$.
Suppose $u \in V(Q_{t_2})$.
Since $u \in X_1$ and $V(Q_{t_2})-X_1 \neq \emptyset$, $Q_{t_2}$ is a nontrivial strong component and so $u$ has an out-neighbor in $Q_{t_2}-X_1$.
Therefore $u \to v_{t_2}$ for some vertex $v_{t_2} \in V(Q_{t_2})-X_1$ in both cases.
Since $v_{t_1} \in X_1$, $u \to v_{t_2} \to v_{t_1} \to U_{3}^{(s)}$.
  If $j=2$, then $u \to v_{t_1} \to U_{3}^{(s)}$ and $v \to U_{2}^{(s)} \to U_{3}^{(s)}$; if $j=3$, then $u \to v_{s-1} \to U_{1}^{(s)}$ and $v \to U_{3}^{(s)} \to U_{1}^{(s)}$.
Therefore, in each case, $u$ and $v$ have an $m$-step common prey for some positive integer $m$.
Thus, if $u \in X_j$ and $v \in U_{j+2}^{(s)}$ for some $j \in \{1,2,3\}$, then $u$ and $v$ are adjacent in $G$ by ($\ast$).
Therefore every vertex in $\bigcup_{i=1}^{t_2}{V(Q_{i})}$ and every vertex in $ V(Q_s)$ are adjacent in $G$.
Hence, by (a), (b), (c1), (d1), (c2), and (d2), $G \cong G_3$ where
\[
K^{(1)}=K\left[\bigcup_{i=1}^{t_2}V(Q_i)\right], \quad
K^{(3)}=K\left[\bigcup_{i=t_2+1}^{t_1}V(Q_i)\right], \quad
K^{(4)}=K\left[\bigcup_{i=t_1+1}^{s-1}V(Q_i)\right],
\]
\[
K^{(5)}=K\left[U_{1}^{(s)}\right], \quad K^{(6)}=K\left[U_{2}^{(s)}\right], \quad \mbox{and} \quad K^{(7)}=K\left[U_{3}^{(s)}\right].\qedhere
\]
\end{proof}

\begin{Thm}\label{thm:unusual}
Suppose that a multipartite tournament $D$ has ordered strong components $Q_1,Q_2,\ldots, Q_s$ and $\kappa(Q_s)=3$ for some integer $s \ge 2$.
  Then the limit of the graph sequence $\{C^m(D)\}_{m=1}^{\infty}$ is isomorphic to $G_3$ given in Figure~\ref{fig:graphs}.
  Further, if $V(Q_{s-1}) \cup V(Q_s)$ does not induce an unusual digraph, then $K^{(2)}$ does not exist.
\end{Thm}

\begin{proof}
Let $G$ be the limit of the graph sequence $\{C^m(D)\}_{m=1}^{\infty}$.
By Lemma~\ref{lem:clique}(i), $V(D) - V(Q_s)$ forms a clique in $G$.
In addition, by Corollary~\ref{Prop:converge}, $V(Q_s)$ induces $\bigcup_{i=1}^{3}{K{\left[U_{i}^{(s)}\right]}}$ in $G$.

{\it Case 1.} The digraph induced by $V(Q_{s-1}) \cup V(Q_s)$ is unusual.
Then $\kappa(Q_{s-1}) = 3$.
By Lemma~\ref{lem:clique}(iii), each vertex in $V(D) - \left(V(Q_{s-1}) \cup V(Q_s)\right)$ and each vertex in $V(Q_s)$ are adjacent in $G$.
Accordingly, we only need to figure out the adjacency of each vertex in $V(Q_{s-1})$ and each vertex in $V(Q_s)$.
To this end, take a vertex $x_i$ in $U_{i}^{(s-1)}$ and a vertex $y_j$ in $U_{j}^{(s)}$ for each $i,j \in \{1,2,3\}$.
 Let $X_i$ be the partite set of $D$ including $U_i^{(s)}$ for each $i \in \{1,2,3\}$.
   Without loss of generality, we may assume $U_{i}^{(s-1)} \subseteq X_i$ for each $i \in \{1,2,3\}$.
  For a nonnegative integer $a$ and for each $i=1,2,3$, we identify the subscripts $3a+i$ with $i$ for $U^{(s)}$, $U^{(s-1)}$, and $X$.
  If $j=i$ (resp.\ $j=i+1$), then each vertex in $U_{i+1}^{(s)}$ (resp.\ $U_{i+2}^{(s)}$) is a common prey of $x_i$ and $y_j$.
  Hence, by ($\ast$), $x_i$ and $y_j$ are adjacent in $G$ if $j=i$ or $i+1$.

   Suppose to the contrary that $x_i$ and $y_j$ have an $m$-step common prey $z$ for some positive integer $m$ where $j=i+2$.
  We note that $y_j$ belongs to the last strong component $Q_s$ and $U_{1}^{(s)},U_{2}^{(s)}$, and $U_{3}^{(s)}$ are the sets of imprimitivity of $Q_s$.
  Therefore
  \begin{equation}\label{eqn:z}
    z \in U_{i+2+m}^{(s)}
  \end{equation}
     and there is an $(x_i,z)$-directed walk $W$ of length $m$ in $D$.
  Let $W=w_0 \to w_1 \to \cdots \to w_m$ where $w_0 = x_i$ and $w_m =z$.
  Since $x_i \in V(Q_{s-1})$ and $z \in V(Q_s)$, there is an (by the way, the only) arc $(w_t,w_{t+1})$ on $W$ such that $w_t \in V(Q_{s-1})$ and $w_{t+1} \in V(Q_s)$ for some $t \in \{0,1,\ldots, m-1\}$.
  Then $w_t$ and $w_{t+1}$ belong to distinct partite sets.
  Since $w_0 \in U_{i}^{(s-1)}$, $w_t \in U_{i+t}^{(s-1)}$.
  Therefore $w_{t+1}$ belongs to $U_{i+t+1}^{(s)}$ or $U_{i+t+2}^{(s)}$.
  Then $w_m$ belongs to $U_{i+m}^{(s)}$ or $U_{i+m+1}^{(s)}$ which contradicts \eqref{eqn:z}.
  Since $x_i$ and $y_j$ were arbitrarily chosen, we may conclude that $j \equiv i+2 \pmod 3$
   \begin{tabbing}
       $\Leftrightarrow$ \= for any positive integer $m$, no vertex in $U_{i}^{(s-1)}$ and no vertex in $U_{j}^{(s)}$ have an $m$-step\\
        \> common prey \\
     $\Leftrightarrow$ \> there is no edge between $U_{i}^{(s-1)}$ and $U_{j}^{(s)}$ in $G$.
   \end{tabbing}
Hence $G \cong G_3$ where
\[
K^{(1)}=K\left[\bigcup_{i=1}^{s-2}V(Q_i)\right], \quad K^{(2)}=K\left[U_{1}^{(s-1)}\right], \quad K^{(3)}=K\left[U_{3}^{(s-1)}\right], K^{(4)}=K\left[U_{2}^{(s-1)}\right],
\]
\[
K^{(5)}=K\left[U_{1}^{(s)}\right], \quad K^{(6)}=K\left[U_{2}^{(s)}\right], \quad \mbox{and} \quad K^{(7)}=K\left[U_{3}^{(s)}\right]
\]
if $V(Q_{s-1}) \cup V(Q_s)$ induces an unusual digraph.

{\it Case 2.} The digraph induced by $V(Q_{s-1}) \cup V(Q_s)$ is not unusual.
 If $Q_{s-1}$ is trivial, then $G  \cong G_3$ by Theorem~\ref{thm:s-1trivial}.
 If $Q_{s-1}$ is nontrivial, then each vertex in $V(D) - V(Q_s)$ and each vertex in $V(Q_s)$ are adjacent in $G$ by (ii) and (iii) of Lemma~\ref{lem:clique}.
 Hence $G=G_3$ where $K^{(2)}$, $K^{(3)}$ and $K^{(4)}$ do not exist and
\[
K^{(1)}=K\left[\bigcup_{i=1}^{s-1}V(Q_i)\right], \quad K^{(5)}=K\left[U_{1}^{(s)}\right], \quad K^{(6)}=K\left[U_{2}^{(s)}\right], \quad \mbox{and} \quad K^{(7)}=K\left[U_{3}^{(s)}\right]
\] if $V(Q_{s-1}) \cup V(Q_s)$ induces a digraph that is not an unusual digraph and $Q_{s-1}$ is nontrivial.
 \end{proof}

\section{A proof of Theorem~\ref{thm:last}}\label{sec:5}
\begin{proof}[Proof of Theorem~\ref{thm:last}]
Suppose that $Q_s$ is nontrivial.
Since $Q_s$ is a strongly connected multipartite tournament, $\kappa(Q_s)\le 4$ by Proposition~\ref{prop:index}.
  If $\kappa(Q_s)=1$, then $G=K[V(D)]$ by Proposition~\ref{prop:lastprimitive}.
  Thus it remains to consider the cases $\kappa(Q_s)=2,3,4$.

  Suppose $s=1$.
  Then, by Theorem~\ref{prop:strong}, the following are true:
  if $\kappa(Q_s)=2$, then  $G \cong G_1$  where $K^{(1)}$ does not exist;
  if $\kappa(Q_s)=3$, then  $G \cong G_3$  where $K^{(i)}$ does not exist for each positive integer $i \le 4$;
  if $\kappa(Q_s)=4$, then  $G \cong  G_2$  where $K^{(i)}$ does not exist for each positive integer $i \le 3$.

  Now suppose $s \ge 2$.
  If $\kappa(Q_s)=2$ (resp.\ $\kappa(Q_s)=4$), then $G\cong G_1$ (resp.\ $G \cong G_2$) by Theorems~\ref{thm:bipartite} and  \ref{Thm:tritrinontri}.
  If $\kappa(Q_s)=3$, then $G \cong  G_3$ by Theorem~\ref{thm:unusual}.
\end{proof}

\section{Disclosure statement}
No potential conflict of interest was reported by the authors.

\section{Funding}
This research was supported by the National Research Foundation of Korea(NRF)  (NRF-2022R1A2C1009648 and NRF-2017R1E1A1A03070489) funded by the Korea government(MSIP).


\begin{thebibliography}{10}

\bibitem{belmont2011complete}
E.~Belmont.
\newblock A complete characterization of paths that are $m$-step competition
  graphs.
\newblock {\em Discrete Appl Math}, 159(14):1381--1390, 2011.

\bibitem{bondy1976diconnected}
J.~A. Bondy.
\newblock Diconnected orientations and a conjecture of las vergnas.
\newblock {\em Journal of the London Mathematical Society}, 2(2):277--282,
  1976.

\bibitem{brualdi1991combinatorial}
R.~A. Brualdi, H.~J. Ryser, et~al.
\newblock {\em Combinatorial matrix theory}, volume~39.
\newblock Springer, 1991.

\bibitem{cho2004competition}
H.~H. Cho and H.~K. Kim.
\newblock Competition indices of digraphs.
\newblock In {\em Proceedings of workshop in combinatorics}, volume~99, pages
  96--107, 2004.

\bibitem{cho2013competition}
H.~H. Cho and H.~K. Kim.
\newblock The competition index of a nearly reducible boolean matrix.
\newblock {\em Bulletin of the Korean Mathematical Society}, 50(6):2001--2011,
  2013.

\bibitem{cho2000m}
H.~H. Cho, S.-R. Kim, and Y.~Nam.
\newblock The $m$-step competition graph of a digraph.
\newblock {\em Discrete Appl Math}, 105(1):115--127, 2000.

\bibitem{cohen1968interval}
J.~E. Cohen.
\newblock Interval graphs and food webs: a finding and a problem.
\newblock {\em RAND Corporation Document}, 17696, 1968.

\bibitem{eoh2020m}
S.~Eoh, S.-R. Kim, and H.~Yoon.
\newblock On $m$-step competition graphs of bipartite tournaments.
\newblock {\em Discrete Appl Math}, 283:199--206, 2020.

\bibitem{helleloid2005connected}
G.~T. Helleloid.
\newblock Connected triangle-free $m$-step competition graphs.
\newblock {\em Discrete Appl Math}, 145(3):376--383, 2005.

\bibitem{ho2005m}
W.~Ho.
\newblock The $m$-step, same-step, and any-step competition graphs.
\newblock {\em Discrete Appl Math}, 152(1):159--175, 2005.

\bibitem{jung2022competition}
J.-H. Jung, S.-R. Kim, and H.~Yoon.
\newblock Competition periods of multipartite tournaments.
\newblock {\em Linear and Multilinear Algebra},
  https://doi.org/10.1080/03081087.2022.2038057.

\bibitem{kim2010generalized}
H.~K. Kim.
\newblock Generalized competition index of a primitive digraph.
\newblock {\em Linear algebra and its applications}, 433(1):72--79, 2010.

\bibitem{kim2015characterization}
H.~K. Kim.
\newblock Characterization of irreducible boolean matrices with the largest
  generalized competition index.
\newblock {\em Linear Algebra Appl}, 466:218--232, 2015.

\bibitem{park2011m}
B.~Park, J.~Y. Lee, and S.-R. Kim.
\newblock The $m$-step competition graphs of doubly partial orders.
\newblock {\em Appl Math Lett}, 24(6):811--816, 2011.

\bibitem{park2013matrix}
W.~Park, B.~Park, and S.-R. Kim.
\newblock A matrix sequence $\{\Gamma({\MakeUppercase{a}}^m)\}_{m=1}^{\infty}$
  might converge even if the matrix ${\MakeUppercase{a}}$ is not primitive.
\newblock {\em Linear Algebra Appl}, 438(5):2306--2319, 2013.

\bibitem{zhao2009note}
Y.~Zhao and G.~J. Chang.
\newblock Note on the $m$-step competition numbers of paths and cycles.
\newblock {\em Discrete Appl Math}, 157(8):1953--1958, 2009.

\end{thebibliography}

\end{document}